\documentclass[11pt,reqno]{amsart}
\usepackage{url}
\usepackage{amssymb}

\newtheorem{thm}{Theorem}

\newtheorem{lem}{Lemma}

\theoremstyle{definition}

\theoremstyle{remark}

\begin{document}

\title[New Congruences of Partitions With Odd Parts Distinct]
 {New Congruences of Partitions \\ With Odd Parts Distinct}

\author{ Liuquan Wang }

\address{Department of Mathematics, National University of Singapore, Singapore, 119076}

\email{wangliuquan@u.nus.edu; mathlqwang@163.com}

\date{May 7, 2014}


\dedicatory{}


\begin{abstract}
 Let $\mathrm{pod}(n)$ denote  the number of partitions of $n$ with odd parts distinct, and  ${{r}_{k}}(n)$ be the number of representations of $n$ as sum of $k$ squares. We find the following two arithmetic relations: for any integer $n\ge 0$,
\[\mathrm{pod}(3n+2)\equiv 2{{(-1)}^{n+1}}{{r}_{5}}(8n+5) \pmod{9}, \]
and
\[\mathrm{pod}(5n+2)\equiv 2{{(-1)}^{n}}{{r}_{3}}(8n+3) \pmod{5}.\]
From which we deduce many interesting congruences including the following two infinite families of Ramanujan-type congruences: for $a \in \{11, 19\}$ and any integers $\alpha \ge 1$ and $n \ge 0$, we have
\[\mathrm{pod}\Big({{5}^{2\alpha +2}}n+\frac{a \cdot {{5}^{2\alpha +1}}+1}{8}\Big)\equiv 0  \pmod{5}.\]
\end{abstract}

\maketitle

\section{{Introduction}}
Let $\psi(q)$ be  one of Ramanujan's theta functions, namely
\[\psi (q)=\sum\limits_{n=0}^{\infty }{{{q}^{n(n+1)/2}}}=\frac{({{q}^{2}};{{q}^{2}})_{\infty }^{2}}{{{(q;q)}_{\infty }}}.\]
We denote by $\mathrm{pod}(n)$  the number of partitions of $n$  with odd parts  distinct. The generating function of $\mathrm{pod}(n)$ is
\begin{equation}\label{gen}
\sum\limits_{n=0}^{\infty }{\mathrm{pod}(n){{q}^{n}}}=\frac{(-q;{{q}^{2}})_{\infty }}{({{q}^{2}};{{q}^{2}})_{\infty }}=\frac{1}{\psi {{(-q)}}}.
\end{equation}

The arithmetic properties of $\mathrm{pod}(n)$ were first studied by  Hirschhorn and Sellers \cite{hisc} in 2010. They obtained some interesting congruences involving the following infinite family of Ramanujan-type congruences: for any integers $\alpha \ge 0$ and $n \ge 0$,
\[\mathrm{pod}\Big({{3}^{2\alpha +3}}n+\frac{23\times {{3}^{2\alpha +2}}+1}{8}\Big)\equiv 0 \pmod{3}.\]

Later on Radu and Sellers \cite{radu} obtained other deep congruences for $\mathrm{pod}(n)$ modulo $5$ and $7$, such as
\[\mathrm{pod}(135n+8)\equiv \mathrm{pod}(135n+107)\equiv \mathrm{pod}(135n+116)\equiv 0 \pmod{5}, \quad \text{\rm{and}}\]
\[\mathrm{pod}(567n+260)\equiv \mathrm{pod}(567n+449)\equiv 0 \pmod{7}.\]

For nonnegative integers $n$ and $k$,  let ${{r}_{k}}(n)$ (resp. ${{t}_{k}}(n)$) denote the number of representations of $n$ as sum of $k$ squares (resp. triangular numbers). In 2011, based on the generating function of $\mathrm{pod}(3n+2)$ found in \cite{hisc}, Lovejoy and Osburn discovered the following arithmetic relation:
\begin{equation}\label{osburn}
\mathrm{pod}(3n+2)\equiv {{(-1)}^{n}}{{r}_{5}}(8n+5) \pmod{3}.
\end{equation}

Following their steps, we will present some new congruences modulo 5 and 9 for $\mathrm{pod}(n)$.

Firstly, we find that  (\ref{osburn}) can be improved to a congruence modulo  9.
\begin{thm}\label{mod9}
For any integer $n\ge 0$, we have
\[\mathrm{pod}(3n+2)\equiv 2{{(-1)}^{n+1}}{{r}_{5}}(8n+5) \pmod{9}.\]
\end{thm}
The following result will be a consequence of Theorem \ref{mod9} upon invoking some properties of $r_{5}(n)$.
\begin{thm}\label{modcor}\label{mod9cor}
Let $p\ge 3$ be a prime, and $N$ be a positive integer such that $pN \equiv 5$ \text{\rm{(mod 8)}}. Let $\alpha$ be any nonnegative integer.\\
(1)  If $p\equiv 1$ \text{\rm{(mod 3)}}, then
\[\mathrm{pod}\Big(\frac{3{{p}^{6\alpha +5}}N+1}{8}\Big)\equiv 0 \pmod{3},\]
and
\[\mathrm{pod}\Big(\frac{3{{p}^{18\alpha +17}}N+1}{8}\Big)\equiv 0  \pmod{9}.\]
(2) If $p\equiv 2$ \text{\rm{(mod 3)}}, then
\[\mathrm{pod}\Big(\frac{3{{p}^{4\alpha +3}}N+1}{8}\Big)\equiv 0 \pmod{9}.\]
\end{thm}

Secondly, with the same method used in proving Theorem \ref{mod9}, we can establish a similar congruence for $\mathrm{pod}(n)$ modulo 5.
\begin{thm}\label{mod5}
For any integer $n\ge 0$, we have
\[\mathrm{pod}(5n+2)\equiv 2{{(-1)}^{n}}{{r}_{3}}(8n+3) \pmod{5}.\]
\end{thm}
Some miscellaneous congruences can be deduced from this theorem.

\begin{thm}\label{mod5cor}
For any integers $n \ge 0$ and $\alpha \ge 1$, we have
\[\mathrm{pod}\Big({{5}^{2\alpha +2}}n+\frac{11\cdot {{5}^{2\alpha +1}}+1}{8}\Big)\equiv 0 \pmod{5},\]
and
\[\mathrm{pod}\Big({{5}^{2\alpha +2}}n+\frac{19\cdot {{5}^{2\alpha +1}}+1}{8}\Big)\equiv 0 \pmod{5}.\]
\end{thm}
\begin{thm}\label{cubemod5}
Let $p\equiv 4$ \text{\rm{(mod 5)}} be a prime, and $N$ be a positive integer which is coprime to $p$  such that $pN \equiv 3$ \text{\rm{(mod 8)}}. We have
\[\mathrm{pod}\Big(\frac{5{{p}^{3}}N+1}{8}\Big)\equiv 0  \pmod{5}.\]
\end{thm}
For example, let $p=19$ and $N=8n+1$ where $n\ge 0$ and $n  \not\equiv 7$ (mod 19), we have
\[\mathrm{pod}(34295n+4287)\equiv 0 \pmod{5}.\]

\begin{thm}\label{2case}
Let $p\ge 3$ be a prime, and $N$ be a positive integer which is not divisible by $p$ such that $pN \equiv 3$ \text{\rm{(mod 8)}}. Let $\alpha $ be any nonnegative integer. \\
(1) If $p\equiv 1$ \text{\rm{(mod 5)}}, we have
\[\mathrm{pod}\Big(\frac{5 {{p}^{10\alpha +9}}N+1}{8}\Big)\equiv 0 \pmod{5}.\]
(2) If $p\equiv 2,3,4$ \text{\rm{(mod 5)}}, we have
\[\mathrm{pod}\Big(\frac{5 {{p}^{8\alpha +7}}N+1}{8}\Big)\equiv 0 \pmod{5}.\]
\end{thm}

\section{Preliminaries }
\begin{lem}\label{palpha}
(Cf. \cite[Lemma 1.2]{radu}.)
Let $p$ be a prime and $\alpha $ be a positive integer. Then
\[(q;q)_{\infty }^{{{p}^{\alpha }}}\equiv ({{q}^{p}};{{q}^{p}})_{\infty }^{{{p}^{\alpha -1}}} \pmod{{{p}^{\alpha }}}.\]
\end{lem}

\begin{lem}\label{t4t8}
For any prime $p \ge 3$,  we have
\[{{t}_{4}}\Big(pn+\frac{p-1}{2}\Big)\equiv {{t}_{4}}(n) \pmod{p},  \quad {{t}_{8}}(pn+p-1)\equiv {{t}_{8}}(n) \pmod{p^3}.\]
\end{lem}
\begin{proof}
By \cite[Theorem  3.6.3]{Bruce}, we know ${{t}_{4}}(n)=\sigma (2n+1).$
For any positive integer $N$, we have
\[\sigma (N)=\sum\limits_{d|N, \, p|d}{d}+\sum\limits_{d|N, \,p \nmid d}{d}\equiv \sum\limits_{d|N, \, p \nmid d}{d} \pmod{p}.\]
Let $N=2n+1$ and $N=p(2n+1)$ respectively.  It is easy to deduce that $\sigma (p(2n+1))\equiv \sigma (2n+1)$ (mod $p$). This clearly implies the first congruence.

From \cite[equation (3.8.3) in page 81]{Bruce}, we know
 \[{{t}_{8}}(n)=\sum\limits_{\begin{smallmatrix}
 d|(n+1) \\
 d \, \mathrm{odd}
\end{smallmatrix}}{{{\Big(\frac{n+1}{d}\Big)}^{3}}}.  \]
By a similar argument we can prove the second congruence.
\end{proof}
\begin{lem}\label{rsrelate}
(Cf. \cite{relation}.) For $1\le k\le 7$, we have
\[{{r}_{k}}(8n+k)={{2}^{k}}\Big(1+\frac{1}{2}\binom{k}{4}\Big){{t}_{k}}(n).\]
\end{lem}

\begin{lem}\label{r5relation}
(Cf. \cite{Cooper}.)
Let $p\ge 3$ be a prime and $n$ be a positive integer such that ${{p}^{2}} \nmid n$. For any integer $\alpha \ge 0$, we have
\[{{r}_{5}}({{p}^{2\alpha }}n)=\Bigg(\frac{{{p}^{3\alpha +3}}-1}{{{p}^{3}}-1}-p\Big(\frac{n}{p}\Big)\frac{{{p}^{3\alpha }}-1}{{{p}^{3}}-1}\Bigg){{r}_{5}}(n),\]
where $(\frac{\cdot }{p})$ denotes the Legendre symbol.
\end{lem}

\begin{lem}\label{r3relation}
(Cf. \cite{3square}.)
Let $p \ge 3$ be a prime. For any  integers $n \ge 1$  and $\alpha \ge 0$,  we have
\[{{r}_{3}}({{p}^{2 \alpha }}n)=\Bigg(\frac{{{p}^{\alpha +1}}-1}{p-1}-\Big(\frac{-n}{p}\Big)\frac{{{p}^{\alpha }}-1}{p-1}\Bigg){{r}_{3}}(n)-p\frac{{{p}^{\alpha }}-1}{p-1}{{r}_{3}}(n/{{p}^{2}}).\]
Here we take ${{r}_{3}}(n/{{p}^{2}})=0$ unless  ${{p}^{2}} | n$ .
\end{lem}

\section{Proofs of the Theorems}
\begin{proof}[Proof of Theorem \ref{mod9}]
Let $p=3$ in Lemma \ref{t4t8}. We deduce that ${{t}_{8}}(3n+2)\equiv {{t}_{8}}(n)$ (mod 9).
By (\ref{gen}) we have
\[\psi {{(q)}^{9}}\sum\limits_{n=0}^{\infty }{\mathrm{pod}(n){{(-q)}^{n}}}=\psi {{(q)}^{8}}=\sum\limits_{n=0}^{\infty }{{{t}_{8}}(n){{q}^{n}}}.\]
By Lemma \ref{palpha} we obtain $\psi {{(q)}^{9}}\equiv  {\psi ({{q}^{3}})}^3$ (mod 9). Hence
\[{\psi (q^3)}^3\sum\limits_{n=0}^{\infty }{\mathrm{pod}(n){{(-q)}^{n}}} \equiv \sum\limits_{n=0}^{\infty }{{{t}_{8}}(n){{q}^{n}}} \pmod{9}.\]
If we extract those terms of the form ${{q}^{3n+2}}$ on both sides, we obtain
\[{\psi (q^3)}^3 \sum\limits_{n=0}^{\infty }{\mathrm{pod}(3n+2){{(-q)}^{3n+2}}}\equiv \sum\limits_{n=0}^{\infty }{{{t}_{8}}(3n+2){{q}^{3n+2}}} \pmod{9}.\]
Dividing both sides by ${{q}^{2}}$, then replacing ${{q}^{3}}$ by $q$, we get
\[\psi {{(q)}^{3}}\sum\limits_{n=0}^{\infty }{\mathrm{pod}(3n+2){{(-q)}^{n}}}\equiv \sum\limits_{n=0}^{\infty }{{{t}_{8}}(3n+2){{q}^{n}}}\equiv \sum\limits_{n=0}^{\infty }{{{t}_{8}}(n){{q}^{n}}}=\psi {{(q)}^{8}} \pmod{9}.\]
Hence
\[\sum\limits_{n=0}^{\infty }{\mathrm{pod}(3n+2){{(-q)}^{n}}}\equiv \psi {{(q)}^{5}}\equiv \sum\limits_{n=0}^{\infty }{{{t}_{5}}(n){{q}^{n}}} \pmod{9}.\]
Comparing the coefficients of $q^{n}$ on both sides, we deduce that $\mathrm{pod}(3n+2)\equiv {{(-1)}^{n}}{{t}_{5}}(n)$ (mod 9).

Let $k=5$ in Lemma \ref{rsrelate}. We obtain ${{t}_{5}}(n)={{r}_{5}}(8n+5)/112$,  and from this the theorem follows.
\end{proof}

\begin{proof}[Proof of Theorem \ref{mod9cor}]
(1)  Let $n=pN$ in Lemma \ref{r5relation}, and then  we replace $\alpha $ by $3\alpha +2$. Since
\[\frac{{{p}^{9\alpha +9}}-1}{{{p}^{3}}-1}=1+{{p}^{3}}+\cdots +{{p}^{3(3\alpha +2)}}\equiv 0 \pmod{3},\]
we deduce that ${{r}_{5}}({{p}^{6\alpha +5}}N)\equiv 0$ (mod 3).

Let $n=\frac{{{p}^{6\alpha +5}}N-5}{8}$ in  Theorem \ref{mod9}. We deduce that $\mathrm{pod}(\frac{3{{p}^{6\alpha +5}}N+1}{8})\equiv 0$ (mod 3).

Similarly, let $n=pN$ in Lemma \ref{r5relation} and we replace $\alpha $ by $9\alpha +8$.  Since $p\equiv 1$ (mod 3) implies
${{p}^{3}}\equiv 1$ (mod 9), we have
\[\frac{{{p}^{27\alpha +27}}-1}{{{p}^{3}}-1}=1+{{p}^{3}}+\cdots +{{p}}^{3(9\alpha +8)}\equiv 0 \pmod{9}.\]
Hence ${{r}_{5}}({{p}^{18\alpha +17}}N)\equiv 0$ (mod 9).

Let $n=\frac{{{p}^{18\alpha +17}}N-5}{8}$ in  Theorem \ref{mod9}. We deduce that $\mathrm{pod}(\frac{3{{p}^{18\alpha +17}}N+1}{8})\equiv 0$ (mod 9).

(2) Let $n=pN$ in Lemma \ref{r5relation}, and then we replace $\alpha $ by $2\alpha +1$. Note that $p\equiv 2$ (mod 3) implies ${{p}^{3}}\equiv -1$ (mod 9). Since ${{p}^{6\alpha +6}}\equiv 1$ (mod 9), we have ${{r}_{5}}({{p}^{4\alpha +3}}N)\equiv 0$ (mod 9).

Let $n=\frac{{{p}^{4\alpha +3}}N-5}{8}$ in Theorem \ref{mod9}. We complete our proof.
\end{proof}

\begin{proof}[Proof of Theorem \ref{mod5}]
Let $p=5$ in Lemma \ref{t4t8}. We deduce that ${{t}_{4}}(5n+2)\equiv {{t}_{4}}(n)$ (mod 5).

By (\ref{gen}) we have
\[\psi {{(q)}^{5}}\sum\limits_{n=0}^{\infty }{\mathrm{pod}(n){{(-q)}^{n}}}=\psi {{(q)}^{4}}=\sum\limits_{n=0}^{\infty }{{{t}_{4}}(n){{q}^{n}}}.\]
By Lemma \ref{palpha} we obtain $\psi {{(q)}^{5}}\equiv  \psi ({{q}^5})$ (mod 5). Hence
\[\psi {{(q^{5})}}\sum\limits_{n=0}^{\infty }{\mathrm{pod}(n){{(-q)}^{n}}} \equiv \sum\limits_{n=0}^{\infty }{{{t}_{4}}(n){{q}^{n}}} \pmod{5}.\]
If we extract those terms of the form ${{q}^{5n+2}}$ on both sides, we obtain
\[\psi (q^5) \sum\limits_{n=0}^{\infty }{\mathrm{pod}(5n+2){{(-q)}^{5n+2}}}\equiv \sum\limits_{n=0}^{\infty }{{{t}_{4}}(5n+2){{q}^{5n+2}}} \pmod{5}.\]
Dividing both sides by ${{q}^{2}}$, and then replacing ${{q}^{5}}$ by $q$,  we get
\[\psi (q)\sum\limits_{n=-\infty }^{\infty }{\mathrm{pod}(5n+2){{(-q)}^{n}}}\equiv \sum\limits_{n=0}^{\infty }{{{t}_{4}}(5n+2){{q}^{n}}}\equiv \sum\limits_{n=0}^{\infty }{{{t}_{4}}(n){{q}^{n}}}=\psi {{(q)}^{4}} \pmod{5}.\]
Hence we have
\[\sum\limits_{n=0}^{\infty }{\mathrm{pod}(5n+2){{(-q)}^{n}}
}\equiv \psi {{(q)}^{3}}= \sum\limits_{n=0}^{\infty }{{{t}_{3}}(n){{q}^{n}}} \pmod{5}.\]
Comparing the coefficients of $q^{n}$ on both sides, we deduce that $\mathrm{pod}(5n+2)\equiv {{(-1)}^{n}}{{t}_{3}}(n)$ (mod 5).

Let $k=3$ in Lemma \ref{rsrelate}. We obtain ${{t}_{3}}(n)={{r}_{3}}(8n+3)/8$,  from which the theorem follows.
\end{proof}

\begin{proof}[Proof of Theorem \ref{mod5cor}]
Let $p=5$ and $n=5m+r$ ($r\in \{1,4\}$) in Lemma \ref{r3relation}. Since $\big(\frac{-r}{5}\big)=1$, we deduce that ${{r}_{3}}({{5}^{2\alpha }}(5m+r))\equiv 0$ (mod 5) for any integer $\alpha \ge 1$.

Let $n=\frac{{{5}^{2\alpha }}(40m+a)-3}{8}$ ($a\in \{11,19\}$). By Theorem \ref{mod5}, we have
\[{{r}_{3}}(8n+3)={{r}_{3}}({{5}^{2\alpha }}(40m+a))\equiv 0 \pmod{5}.\]
Hence
\[\mathrm{pod}\Big({{5}^{2\alpha +2}}m+\frac{a\cdot {{5}^{2\alpha +1}}+1}{8}\Big)=\mathrm{pod}(5n+2)\equiv 2{{(-1)}^{n}}{{r}_{3}}(8n+3)\equiv 0 \pmod{5}.\]
\end{proof}

\begin{proof}[Proof of Theorem \ref{cubemod5}]
Let $\alpha =1$ and $n=pN$ in Lemma \ref{r3relation}. We have
\[{{r}_{3}}({{p}^{3}}N)=(1+p){{r}_{3}}(pN)\equiv 0 \pmod{5}.\]
Let $n=\frac{{{p}^{3}}N-3}{8}$ in Theorem \ref{mod5}, we have
\[\mathrm{pod}\Big(\frac{5{{p}^{3}}N+1}{8}\Big)=\mathrm{pod}(5n+2)\equiv 2{{(-1)}^{n}}{{r}_{3}}(8n+3)=2{{(-1)}^{n}}{{r}_{3}}({{p}^{3}}N)\equiv 0 \pmod{5}.\]
\end{proof}

\begin{proof}[Proof of Theorem \ref{2case}]
(1)  Let $n=pN$ in Lemma \ref{r3relation}, and then we replace $\alpha $ by $5\alpha +4$. We have
\[\frac{{{p}^{5\alpha +5}}-1}{p-1}=1+p+\cdots +{{p}^{5\alpha +4}}\equiv 0 \pmod{5}.\]
Hence ${{r}_{3}}({{p}^{10\alpha +9}}N)\equiv 0$ (mod 5). Let $n=\frac{{{p}^{10\alpha +9}}N-3}{8}$ in Theorem \ref{mod5}. We have
\[\mathrm{pod}\Big(\frac{5{{p}^{10\alpha +9}}N+1}{8}\Big)=\mathrm{pod}(5n+2) \equiv 2{{(-1)}^{n}}{{r}_{3}}({{p}^{10\alpha +9}}N)\equiv 0  \pmod{5}.\]
(2)  Let $n=pN$ in Lemma \ref{r3relation}, and then we replace $\alpha $ by $4\alpha +3$. Since  ${{p}^{4\alpha +4}}\equiv 1$ (mod 5), we deduce that ${{r}_{3}}({{p}^{8\alpha +7}}N)\equiv 0$ (mod 5). Let $n=\frac{{{p}^{8\alpha +7}}N-3}{8}$ in Theorem \ref{mod5}, we have
\[\mathrm{pod}\Big(\frac{5{{p}^{8\alpha +7}}N+1}{8}\Big)=\mathrm{pod}(5n+2)\equiv 2{{(-1)}^{n}}{{r}_{3}}({{p}^{8\alpha +7}}N)\equiv 0 \pmod{5}.\]
\end{proof}

\bibliographystyle{amsplain}

\end{document}